\documentclass[12pt]{amsproc}
\newtheorem{theorem}{\sc Theorem}
\newtheorem{lemma}[theorem]{\sc Lemma}
\newtheorem{proposition}[theorem]{\sc Proposition}
\newtheorem{corollary}[theorem]{\sc Corollary}

\newtheorem{Index Convention}{Index Convention}
\begin{document}
\title{The ranks of central factor and commutator groups}
\author{Leonid A. Kurdachenko}
\address{Department of Algebra, National University of Dnepropetrovsk
72 Gagarin Av., Dnepropetrovsk, Ukraine 49010}
\email{lkurdachenko@i.ua}
\author{Pavel Shumyatsky}
\address{Department of Mathematics, University of Brasilia,
Brasilia-DF, 70910-900 Brazil ,pavel@unb.br}
\email{pavel@mat.unb.br}
\thanks{The work of the second author was supported by CNPq-Brazil}
\begin{abstract} The Schur Theorem says that if $G$ is a group whose center $Z(G)$ has finite index $n$, then the order of the derived group $G'$ is finite and bounded by a number depending only on $n$. In the present paper we show that if $G$ is a finite group such that $G/Z(G)$ has rank $r$, then the rank of $G'$ is $r$-bounded. We also show that a similar result holds for a large class of infinite groups.
\end{abstract}
\maketitle

The famous Schur Theorem says that if $G$ is a group whose center $Z(G)$ has finite index $n$, then the order of the derived group $G'$ is finite and bounded by a number depending only on $n$. Throughout the paper we say that a quantity is $\{a,b,c,\dots\}$-bounded if it is bounded by a number that depends only on the parameters $a,b,c,\dots$. The theorem provides a very useful tool for group theorists. Quite naturally, ever since the theorem was proved the relation between $G/Z(G)$ and $G'$ was in the focus of considerable attention. One rather straightforward generalization of the Schur theorem is that if $G/Z(G)$ is locally finite, then $G'$ is locally finite as well (see \cite[p. 102]{robinson} for example). On the other hand, there is a well-known example, due to Adian, of a torsion-free group $G$ such that $G/Z(G)$ is periodic and even has prime exponent \cite{adian}. Thus, periodicity of $G/Z(G)$ in general does not imply that of $G'$. Using the positive solution to the restricted Burnside problem \cite{ze1,ze2} A. Mann showed that if $G/Z(G)$ is locally finite and has finite exponent $n$, then $G'$ is locally finite and has finite $n$-bounded exponent \cite{mann}. A number of other results in the spirit of the Schur theorem can be found in \cite{polo,kur,dfk}. In the present paper we examine the situation where $G/Z(G)$ has finite rank. Recall that a group $K$ has finite rank $r$ if every finitely generated subgroup of $K$ can be generated by at most $r$ elements. We would like to thank A. Yu. Olshanskii for explaining to us how it can be shown that there exists a group $G$ with the property that $Z(G)$ is a free abelian group of infinite rank while $G$ is a perfect group such that all proper subgroups of $G/Z(G)$ are of prime order. The proof of this fact goes as follows. In Section 27 of \cite{olsh} one finds the construction of the group $G(\infty)$ as a limit of the sequence $\{G(i)\}$. All proper subgroups of $G(\infty)$ have prime order \cite[Theorem 28.1]{olsh}. The combination of Lemma 27.2, Lemma 25.1 and Theorem 31.1(2) of \cite{olsh} guarantees that for the aspherical corepresentation $G(\infty)=F/N$ the group $N/[N,F]$ is a free abelian group of infinite rank. Now arguing as in the proof of Corollary 31.2 we deduce that the Schur multiplier of $G(\infty)$ is a free abelian group of infinite rank. The proof also uses the fact that the group $G(\infty)$ does not coincide with $G(i)$ for any $i=1,2,\dots$. This can be shown by arguing as in the proof of Theorem 19.3 and replacing the reference to Theorem 19.1 by that to Theorem 26.2.

Thus, finiteness of the rank of $G/Z(G)$ in general does not imply that of  $G'$. We will show however that under some reasonable additional hypothesis on the group $G$ we do have a rank version of the Schur theorem.

Our first result provides a rank version of the Schur theorem for finite groups.

\begin{theorem}\label{main} Let $G$ be a finite group such that $G/Z(G)$ has rank $r$. Then the rank of $G'$ is $r$-bounded.
\end{theorem}

The proof of the above theorem depends on the classification of finite simple groups. The Lubotzky-Mann theory of powerful $p$-groups \cite{ddms} plays an important role in the proof as well. From Theorem \ref{main} we deduce a rank version of the Schur theorem for a large class of infinite groups.

A group $G$ is called {\it generalized radical} if it has an ascending series whose quotients are either locally nilpotent or locally finite. Accordingly, a group $G$ is locally generalized radical if every finitely generated subgroup of $G$ is generalized radical.

\begin{theorem}\label{second} Let $G$ be a locally generalized radical group such that $G/Z(G)$ has finite rank $r$. Then the rank of $G'$ is finite and $r$-bounded.
\end{theorem}

In the present paper we make no attempts to write down explicit bounds for the rank of $G'$ in Theorems \ref{main} and \ref{second}. Throughout the paper we use without explicit references the facts that if $r(G)=r$, then every subgroup and every quotient of $G$ has rank at most $r$ and that if $G$ has a normal subgroup $N$ such that $r(N)=r_1$ and $r(G/N)=r_2$, then $r(G)\leq r_1+r_2$.

The next lemma is well-known but for the reader's convenience we include the proof.

\begin{lemma}\label{not new} Let $G$ be a group having a subset $X$ and a normal abelian subgroup $N$ such that $G=\langle N,X\rangle$. Then $[N,G]=\prod_{x\in X}[N,x]$.
\end{lemma}
\begin{proof} Set $K=\prod_{x\in X}[N,x]$. It is clear that $K\leq[N,G]$ so we need to prove the other inclusion. If $K=1$, then $N$ is central in $G$ and therefore $[N,G]=1$. Thus, it is sufficient to show that $K$ is normal in $G$. The subgroup $N$ normalizes $K$ because $N$ is abelian and $K\leq N$. Therefore we only need to show that every element $x$ of $X$ normalizes $K$. It is straightforward that $x$ normalizes every subgroup of $N$ that contains $[N,x]$. Since $[N,x]\leq K$ for all $x\in X$, the lemma follows.
\end{proof}

\begin{lemma}\label{1.0} Let $G$ be a group such that $G/Z(G)$ has finite rank $r$. Let $M$ be a normal abelian subgroup of $G$. Then $[M,G]$ has rank at most $r^2$. 
\end{lemma}
\begin{proof} Set $N=MZ(G)$ and suppose first that $G$ is finitely generated. Then $N$ is a normal abelian subgroup of $G$ such that $G/N$ can be generated by at most $r$ elements. Therefore we can choose elements $x_1,\dots,x_r\in G$ such that $G=\langle N,x_1,\dots,x_r\rangle$. By Lemma \ref{not new} $[N,G]=\prod[N,x_i]$. We note that for any $x\in G$ the map that takes every element $y\in N$ to $y^{-1}y^x$ is a homomorphism of $N$ in $[N,x]$ whose kernel is $C_N(x)$. Since $G/Z(G)$ has rank $r$, it follows that the rank of $N/C_N(x)$ is at most $r$ for all $x\in G$ and hence the rank of $[N,x]$ is at most $r$ as well. The equality $[N,G]=\prod[N,x_i]$ now guarantees that the rank of $[N,G]$ is bounded by $r^2$. Finally, we remark that $[N,G]=[M,G]$ and so in the case where $G$ is finitely generated the lemma follows.

Let us now drop the assumption that $G$ is finitely generated. Suppose that $[M,G]$ has rank at least $r^2+1$. Thus, we can choose elements $y_1,\dots,y_{r^2+1}\in[M,G]$ such that the subgroup $\langle y_1,\dots,y_{r^2+1}\rangle$ cannot be generated by $r^2$ elements. We can also choose a finitely generated subgroup $K$ in $G$ such that $y_1,\dots,y_{r^2+1}\in[M,K]\cap K$. This yields a contradiction since we know that for finitely generated groups the lemma holds. The proof is now complete.
\end{proof}

\begin{lemma}\label{1.1} Let $d$ and $r$ be positive integers and $G$ a group such that $G/Z(G)$ is soluble with derived length $d$ and $r(G/Z(G))=r$. Then $G'$ has finite rank and $r(G')\leq (1/2)dr(r-1)+(d-1)r^2$.
\end{lemma}
\begin{proof} Denote the expression $(1/2)dr(r-1)+(d-1)r^2$ by $N(d,r)$. If the result is false, then $G'$ contains elements $y_1,\dots,y_{N(d,r)+1}$ such that the subgroup $\langle y_1,\dots,y_{N(d,r)+1}\rangle$ cannot be generated by less than $N(d,r)+1$ elements. We can choose a finitely generated subgroup $K$ of $G$ such that $y_1,\dots,y_{N(d,r)+1}\in K'$ and so $K$ provides a counter-example to the lemma. Thus, it is sufficient to prove the lemma for finitely generated groups and so we assume that $G$ is finitely generated. It follows that $G/Z(G)$ can be generated by $r$ elements. In particular we deduce that $G=\langle Z(G),x_1,\dots,x_r\rangle$ for suitable $x_1,\dots,x_r\in G$.

The lemma will be proved by induction on $d$. Suppose first that $d=1$. Then $G$ is nilpotent of class at most two and, since $G=\langle Z(G),x_1,\dots,x_r\rangle$, it is clear that $G'$ is generated by the commutators $[x_i,x_j]$, where $1\leq i<j\leq r$. There are at most $(1/2)r(r-1)$ such commutators and so in the case where $d=1$ the lemma holds.

Now we assume that $d\geq 2$ and that the rank of the derived group of $G'Z(G)$ is at most $N(d-1,r)$. Of course, the derived group of $G'Z(G)$ is precisely $G''$, the second derived group of $G$. We pass to the quotient $G/G''$ and assume that $G$ is metabelian. Then, by Lemma \ref{1.0}, the rank of $[G',G]$ is at most $r^2$. The quotient $G/[G',G]$ is nilpotent of class at most two and so by what we have established above the rank of $G'/[G',G]$ is at most $(1/2)r(r-1)$. Hence, in the case where $G$ is metabelian the rank of $G'$ is at most $(1/2)r(r-1)+r^2$. By the induction hypothesis the rank of $G''$ is at most $N(d-1,r)$ and so we deduce that the rank of $G'$ is at most $(1/2)r(r-1)+r^2+N(d-1,r)=N(d,r)$, as required.
\end{proof}

\begin{lemma}\label{1.2} There exists an integer $s=s(n,r)$ depending only on $n$ and $r$ such that if $G$ is a finite $p$-group of exponent dividing $p^n$ and rank at most $r$, then the order of $G$  is at most $p^s$.
\end{lemma}
\begin{proof} By Theorem 2.13 of \cite{ddms} $G$ has a powerful characteristic subgroup $N$ of index at most $p^{\mu(r)}$, where $\mu(r)$ is a number depending only on $r$. Corollary 2.8 in \cite{ddms} shows that $N$ is a product of at most $r$ cyclic subgroups. Therefore $N$ is of order at most $p^{nr}$ and the lemma follows. 
\end{proof}

\begin{proposition}\label{nilpo} If $G$ is a finite nilpotent group, then the rank of $G'$ is bounded in terms of $r=r(G/Z(G))$ only.
\end{proposition}
\begin{proof} It is clear that without loss of generality we can assume that $G$ is a $p$-group for some prime $p$. Let $n$ be the least number such that $2^n\geq r+2$. Then obviously $G$ does not involve the wreath product $C_p\wr C_{p^n}$. Let $K=G'$ and $H=K^{p^n}$ if $p$ is odd and $H=K^{p^{n+1}}$ if $p=2$. By \cite{shalev92} $H$ is powerful. Let $\overline G=G/H^p$ and $\overline K=K/H^p$. Obviously $\overline K$ has exponent dividing $p^{n+2}$. By the hypothesis $\overline G/Z(\overline G)$ has rank at most $r$ and therefore, by Lemma \ref{1.2}, $\overline K/Z(\overline K)$ has order at most $p^s$ for some $r$-bounded number $s$. Let $d$ denote the derived length of $\overline G$. It is clear that $d\leq s+1$ and in particular $d$ is $r$-bounded. Lemma \ref{1.1} now tells us that the rank, say $t$, of $\overline K$ is $r$-bounded. Since $H^p$ is the Frattini subgroup of $H$ and since $H/H^p$ can be generated by $t$ elements, it follows that $H$ can be generated by $t$ elements. Taking into account that $H$ is powerful, we deduce that $H$ is of rank at most $t$. In particular $H^p$ is of rank at most $t$. Combining this with the fact that $\overline K$ has rank at most $t$ we now conclude that the rank of $K$ is at most 2$t$, as required. 
\end{proof}

A well-known theorem of Zassenhaus says that whenever $F$ is a field, the derived length of any soluble subgroup of $GL_n(F)$ is bounded by a function of $n$ only \cite{zasse}. This will be used in the following lemma. Given a group $G$, we denote by $F(G)$ the Fitting subgroup of $G$.

\begin{lemma}\label{zass} Let $G$ be a finite soluble group of rank $r$. Then the derived length of $G/F(G)$ is $r$-bounded.
\end{lemma}
\begin{proof} Consider an unrefinable normal series $$G=N_1>N_2 > \dots > N_k > N_{k+1}=1$$ in $G$. The factors $N_i/N_{i+1}$ are elementary abelian of rank at most $r$ and so every factor $N_i/N_{i+1}$ can be viewed as a linear space of dimension at most $r$ over some field with $p$ elements. By Zassenhaus' Theorem there is an integer $d$ depending only on $r$ such that every soluble group of automorphisms of $N_i/N_{i+1}$ has derived length at most $d$. Let $K$ be the $d$th derived group of $G$. Then $K$ centralizes every factor of the series $$G=N_1>N_2>\dots >N_k > N_{k+1}=1$$ and hence $K$ is nilpotent. Thus, $K\leq F(G)$ and $G/F(G)$ has derived length at most $d$.
\end{proof}

In the proof of the next lemma we use the well-known corollary of the classification of finite simple groups that $Aut\,S/Inn\,S$, the outer automorphism group of $S$, is soluble for every finite simple group $S$ (this fact is also known under the name of the Schreier Conjecture).  

\begin{lemma}\label{schrei} Let $G$ be a finite group of rank $r$ and assume that $G$ has no nontrivial normal soluble subgroups. Then $G$ has a normal series $$M\leq T\leq G$$ such that $M$ is isomorphic to a direct product of at most $r$ non-abelian simple groups, $T/M$ is soluble and $G/T$ has order at most $r!$.
\end{lemma}
\begin{proof} Let $M$ be the product of all minimal normal subgroups of $G$. Of course, $M=S_1\times S_2\times\dots\times S_k$, where the subgroups $S_1,S_2,\dots,S_k$ are isomorphic with non-abelian simple groups. Since $r(M)\leq r$ and since all subgroups $S_1,S_2,\dots,S_k$ have even order, it follows that $k\leq r$.

Our group $G$ acts by conjugation on $M$ and this action induces a natural homomorphism of $G$ in the symmetric group on $k$ symbols. Let $T$ be the kernel of the homomorphism. In other words, $T$ is the intersection of the normalizers $N_G(S_i)$ for $i=1,2,\dots,k$. Let $U$ be the last term of the derived series of $T$. Thus, $U$ is the intersection of all normal subgroups $N$ of $T$ such that $T/N$ is soluble. For every $i=1,2,\dots,k$ put $T_i=S_iC_T(S_i)$. Then $T/T_i$ embeds into $Aut\,S_i/Inn\,S_i$, which is soluble. Therefore $U\leq T_i$ for all $i=1,2,\dots,k$. Hence, any element in $U$ induces an inner automorphism of $S_i$ for all $i\leq k$. It follows that for every $x\in U$ there exist elements $x_i\in S_i$ such that $$x_1x_2\dots x_kx^{-1}\in C_T(S_i) \text{ for all } i=1,2,\dots,k.$$ Thus, $x_1x_2\dots x_kx^{-1}\in C_T(M)$. Since $G$ has no nontrivial normal soluble subgroups, it follows that $C_T(M)=1$ and therefore $x\in M$. Thus, $U=M$ and $T/M$ is soluble.

Since $G/T$ embeds in the symmetric group on $k$ symbols, it follows that the index of $T$ in $G$ is at most $r!$. The proof is complete.
\end{proof}

\begin{proof}[Proof of Theorem \ref{main}] Recall that $G$ is a finite group such that $G/Z(G)$ has rank $r$. Let $R$ be the maximal normal soluble subgroup of $G$ and $F=F(R)$. Since $R/Z(R)$ has rank at most $r$ and since $F(R/Z(R))=F/Z(R)$, Lemma \ref{zass} shows that the derived length of $R/F$ is $r$-bounded. By Proposition \ref{nilpo} $F'$ has $r$-bounded rank and therefore we can pass to the quotient $G/F'$. Hence, without loss of generality we can assume that the derived length of $R$ is $r$-bounded. Now Lemma \ref{1.1} guarantees that the rank of $R'$ is $r$-bounded and we pass to the quotient $G/R'$. Thus, we will assume that $R$ is abelian. Since $G/R$ has rank $r$, we can choose $r$ elements $a_1,\dots,a_r$ such that $G=\langle R,a_1,\dots,a_r\rangle$. Then, by Lemma \ref{1.0}, the rank of $[R,G]$ is at most $r^2$. Passing to the quotient $G/[R,G]$ assume that $R=Z(G)$.

The structure of $G/R$ is described in Lemma \ref{schrei}. We assume that $G\neq R$ and let $M/R$ be the the product of all minimal normal subgroups of $G/R$. Then $M/R$ is a direct product of $k\leq r$ simple non-abelian groups. For $i=1,\dots,k$ let $S_i$ denote the subgroup of $G$ such that $S_i/R$ is a simple factor of $M/R$. Of course, $M=S_1S_2\dots S_k$. Let $D_i$ be the derived group of $S_i$ for $i=1,2,\dots,k$. Then $D_i$ is a perfect group and $D_i/Z(D_i)$ is a simple group of rank at most $r$. Thus, $D_i$ is a so-called quasisimple group. The Schur multipliers of all simple groups are well-known and can be found in \cite[p. 302--303]{gorenstein}. All of them are abelian groups of rank at most three. Thus, the rank of $D_i$ is at most $r+3$. From this we deduce that the product $D=D_1D_2\dots D_k$ has rank at most $r(r+3)$. Thus, we pass to the quotient $\overline{G}=G/D$. Lemma \ref{schrei} tells us that $\overline{G}$ has a normal soluble subgroup $\overline{T}=T/D$ of index at most $r!$. Precisely in the same way as we have shown above that the rank of $[R,G]$ is $r$-bounded we can now show that the rank of $[\overline{T},\overline{G}]$ is $r$-bounded as well. Passing to the quotient $\overline{G}/[\overline{T},\overline{G}]$, assume that $\overline{T}$ is central. Then the index of $Z(\overline{G})$ in $\overline{G}$ is at most $r!$ and by Schur's theorem the order of the derived group of $\overline{G}$ is $r$-bounded. The proof is now complete.
\end{proof}

A routine inverse limit argument along the lines of \cite{kewe} now shows that if $G$ is a locally finite group such that $G/Z(G)$ has rank $r$, then the rank of $G'$ is $r$-bounded. The following corollary is a little stronger.

\begin{corollary}\label{locfin} Let $G$ be a group such that $G/Z(G)$ is locally finite and has finite rank $r$. Then $G'$ is locally finite and the rank of $G'$ is $r$-bounded.
\end{corollary}
\begin{proof} By \cite[p. 102]{robinson} $G'$ is locally finite and so the above comment shows that $G''$ has finite $r$-bounded rank. Thus, we can pass to the quotient $G/G''$ and assume that $G$ is metabelian. Now Lemma \ref{1.1} tells us that $r(G')\leq r(r-1)+r^2$.
\end{proof}

We will now proceed to establish Theorem \ref{second}. First of all we note that subgroups and quotients of a generalized radical group are generalized radical as well. We also notice that periodic generalized radical groups are locally finite.

 Let $G$ be a group having an ascending series whose factors are either cyclic or periodic. If the number of infinite cyclic factors in the series is finite, we call it $0$-rank of $G$ and denote it by $r_0(G)$. If the number of infinite cyclic factors in the series is infinite, we say that $G$ has infinite $0$-rank. It is not difficult to see that the $0$-rank of $G$ does not depend on the choice of the ascending series whose factors are either cyclic or periodic.

\begin{lemma}\label{solv} Suppose that $G$ is a finitely generated generalized radical group of finite rank $r$ and assume that $G$ has no nontrivial normal periodic subgroups. Then $G$ has a subgroup $L$ of $r$-bounded index which is soluble with $r$-bounded derived length.
\end{lemma}
\begin{proof} It is clear that $r_0(A)\leq r$ for every abelian subgroup $A$ of $G$. By \cite[Theorem 1]{2007} $r_0(G)$ is bounded in terms of $r$ only. Moreover \cite[Theorem A]{2007} implies that $G$ has a normal series $$K\leq L\leq G$$ such that $K$ is torsion-free nilpotent, $L/K$ is torsion-free abelian and $G/L$ is finite of $r$-bounded order. Since $r_0(G)$ is bounded in terms of $r$ only, it follows that the nilpotency class of $K$ is $r$-bounded and hence the derived length of $L$ is $r$-bounded, too.
\end{proof}

We are now ready to prove Theorem \ref{second}.

\begin{proof}[Proof of Theorem \ref{second}] Recall that $G$ is a locally generalized radical group such that $G/Z(G)$ has finite rank $r$. Assume first that $G$ is finitely generated and let $T/Z(G)$ be the product of all periodic subgroups of $G/Z(G)$. Thus $T/Z(G)$ is the maximal normal periodic subgroup of $G/Z(G)$ and so $G/T$ has no nontrivial normal periodic subgroups. By Lemma \ref{solv} $G$ has a subgroup $L$ of $r$-bounded index such that $L/T$ is soluble with $r$-bounded derived length. Since $G$ is a generalized radical group, it follows that all periodic sections of $G$ are locally finite and therefore so is $T/Z(G)$. Corollary \ref{locfin} tells us now that $r(T')$ is $r$-bounded and passing to $G/T'$ we can assume that $T'=1$. In this case $T$ is abelian and so $L$ is soluble with $r$-bounded derived length. Lemma \ref{1.1} can now be applied to deduce that $r(L')$ is $r$-bounded and we can pass to the quotient $G/L'$. Now $L$ is a normal abelian subgroup and so by Lemma \ref{1.0} $[L,G]$ has rank at most $r^2$. Passing to the quotient $G/[L,G]$, assume that $L$ is central. Then the index of $Z(G)$ in $G$ is $r$-bounded and by Schur's theorem the order of the derived group of $G$ is $r$-bounded as well. This proves the theorem in the case where $G$ is finitely generated. In other words, there exists an $r$-bounded number, say $R_0$, such that $r(K')\leq R_0$ whenever a group $K$ satisfies the hypothesis of the theorem and is finitely generated. Suppose now that our group $G$ is not necessarily finitely generated. If $r(G')\geq R_0+1$, we can choose $y_1,\dots,y_{R_0+1}\in G'$ such that the subgroup $\langle y_1,\dots,y_{R_0+1}\rangle$ cannot be generated by $R_0$ elements. We can also choose a finitely generated subgroup $K$ in $G$ such that $y_1,\dots,y_{R_0+1}\in K'$. This yields a contradiction since we know that $r(K')\leq R_0$. The proof is now complete.

\end{proof}


\begin{thebibliography}{99}

\bibitem{adian} S. I. Adian, On some torsion-free groups, Math. USSR-Izv., {\bf 5} (1971), 475–484.

\bibitem{ddms} J. D. Dixon, M. P. F. du Sautoy, A. Mann, D. Segal, {\it
Analytic $p$-adic groups}, Cambridge Univ. Press, 1991.

\bibitem{2007} M. R. Dixon, L. A. Kurdachenko, N. V. Polyakov, On some ranks of
infinite groups, Ricerche Mat., {\bf 56} (2007), 43–-59.

\bibitem{dfk} S. Franciosi, F. de Giovanni, L. A. Kurdachenko, The Schur property and groups with uniform conjugacy classes, J. Algebra {\bf 174} (1995),
823--847.

\bibitem{gorenstein} D. Gorenstein, {\it Finite simple groups: An introduction to their classification}, Plenum Press, New York and London, 1982.

\bibitem{kewe} O. H. Kegel, B. F. A. Wehrfritz, \textit{Locally finite groups}, North-Holland, Amsterdam, 1973.

\bibitem{kur} L. A. Kurdachenko, On groups with minimax conjugacy classes.
"Infinite groups and adjoining algebraic structures", Naukova Dumka
Kiev, 1993, 160–-177.

\bibitem{mann} A. Mann, The exponents of central factor and commutator groups, J. Group Theory {\bf 10} (2007), no. 4, 435--436. 

\bibitem{olsh} A. Yu. Olshanskii. Geometry of defining relations in groups. Kluwer Academic Publishers, 1991.

\bibitem{polo} Ya.D. Polovicky, Groups with extremal classes of conjugate elements, Sibir. Math. J. 5 (1964), 891–-895.

\bibitem{robinson} D. J. S. Robinson, {\it Finiteness Conditions and Generalized Soluble Groups}, Part 1, Springer Verlag, Berlin-New York, 1972.

\bibitem{shalev92} A. Shalev, Characterization of p-Adic Analytic Groups in
Terms of Wreath Products, J. Algebra {\bf 145} (1992), 204--208.

\bibitem{zasse} H. Zassenhaus, Beweis eines Satzes über diskrete Gruppen,
Abh. Math. Sem. Univ. Hamburg, {\bf 12} (1938), 289 -– 312.

\bibitem{ze1} E. Zelmanov, The solution of the Restricted Burnside Problem 
for groups of odd exponent, Math. USSR Izv., {\bf 36} (1991), 41--60.

\bibitem{ze2} E. Zelmanov, The solution of the Restricted Burnside Problem 
for 2-groups, Math. Sb., {\bf 182} (1991), 568--592.

\end{thebibliography}
\end{document}